\documentclass[11pt]{article}

\usepackage{amsmath,amssymb,amsthm}
\usepackage{graphicx}
\usepackage{subfigure}
\usepackage{enumerate}
\usepackage{fullpage}
\usepackage{prettyref}
\usepackage{url}
\usepackage{pstricks}

\usepackage{xcolor}
\definecolor{foocite}{rgb}{0,0.75,0}
\definecolor{foolink}{rgb}{0,0,1}
\definecolor{foourl}{rgb}{1,0,0}
\usepackage[colorlinks=true,citecolor=foocite,linkcolor=foolink,urlcolor=foourl]{hyperref}

\newcommand{\comment}[1]{}

\def\bs{\backslash}
\def\fit{\mathcal{F}}

\newenvironment{packed_enum}{
\begin{itemize}
  \setlength{\itemsep}{1pt}
  \setlength{\parskip}{0pt}
  \setlength{\parsep}{0pt}
}{\end{itemize}}

\newtheorem{theorem}{Theorem}
\newrefformat{theorem}{Theorem~\ref{#1}}
\newrefformat{eq}{Equation~\eqref{#1}}
\newrefformat{chap}{Chapter~\ref{#1}}
\newrefformat{fig}{Figure~\ref{#1}}
\def\xthm[#1][#2][#3]{\newtheorem{#2}[theorem]{#3} \newrefformat{#2}{#3 \ref{#11}}}
\xthm[#][definition][Definition]
\xthm[#][claim][Claim]
\xthm[#][conjecture][Conjecture]
\xthm[#][proposition][Proposition]
\xthm[#][lemma][Lemma]
\xthm[#][example][Example]
\xthm[#][fact][Fact]
\xthm[#][corollary][Corollary]

\newcommand{\answerCommand}{}%
  {\renewcommand{\answerCommand}{#1\\}%
   \noindent\textbf{\answerCommand}%
   }{\\}%

  {\renewcommand{\answerCommand}{#1}%
   \noindent\textbf{\answerCommand}%
   }{\\}%

\title{Diameter Bounds for Planar Graphs}
\author{Radoslav Fulek${}^*$
\and
Filip Mori\'{c}${}^*$
\and David Pritchard\thanks{Ecole Polytechnique F\'ed\'erale de Lausanne. Email:~$\{$\texttt{radoslav.fulek}, \texttt{filip.moric}, \texttt{david.pritchard}$\}$\texttt{@epfl.ch}}
}
\date{}

\begin{document}

\maketitle

\thispagestyle{empty}

\begin{abstract}
The \emph{inverse degree} of a graph is the sum of the reciprocals of the degrees of its vertices. We prove that in any connected planar graph, the diameter is at most $5/2$ times the inverse degree, and that this ratio is tight. To develop a crucial surgery method, we begin by proving the simpler related upper bounds $(4(|V|-1)-|E|)/3$ and $4|V|^2/3|E|$ on the diameter (for connected planar graphs), which are also tight.
\end{abstract}



\section{Introduction}

\comment{By a (simple) {\it graph} $G=(V,E)$ we understand a pair consisting of the finite
set of {\it vertices} $V$ and the finite set of {\it edges} $E$ such that $E\subseteq {V \choose 2}$.

By a {\it path} $P=uPv$ between two vertices $u$ and $v$ in $G$ we understand a sequence $u_0=uu_1\ldots u_p=v$ such
that $u_iu_{i+1}\in E$ and $u_i\not=u_j$, if $i\not=j$. Then $l(P)=n$ denotes  the {\it length} of $P$.
Throughout the paper we assume that our graph $G$ is connected, i.e. there exists a path between any pair of vertices.}

In this paper we examine the relation between ``inverse degree" and diameter in connected planar simple graphs. The diameter $D(G)$ of a graph $G = (V, E)$ is the maximum distance between any pair of vertices, $D := \max_{u, v \in V} dist(u, v),$ where as usual the distance between two vertices is the minimum number of edges on any $u$-$v$ path. The \emph{inverse degree} $r(G)$ is a less well-studied quantity, and is defined equal to the sum of the inverses of the degrees, $r := \sum_{v \in V} d^{-1}(v)$.

\comment{By a distance $d(u,v)$ between two vertices $u,v$ in $G$ we understand the length of a shortest path between $u$ and $v$.
We define the {\it diameter} of $G$ (denoted by $(D(G)$) to be a greatest distance between two vertices of $G$,
i.e. $D(G)=\max_{u,v\in V}d(u,v)$. 
If $v\in V$, we let $d(v)$ denote the number of edges $e$ in $E$ incident to $v$, i.e. the number
of edges $e$ such that $v\in e$, or shortly the degree of $v$.
We define the inverse degree $r(G)$ of a finite graph $G=(V,E)$ as $\sum_{v\in V}1/d(v)$, where $d(v)$ is
the degree of the vertex $v$.

By drawing of a graph $G=(V,E)$ in the plane we understand a representation of $G$ in which each  vertex in $V$
corresponds to a point in the Euclidean plane $\mathbb{R}^2$ and each edge is represented by a Jordan arc joining its endpoints.
By a planar graph $G=(V,E)$ we understand a graph which can be drawn in the plane without edge crossings.
By a plane graph $G$ we understand a drawing of the planar graph $G$ in the plane without edge crossings.
By a face of a plane graph $G$ we understand a connected region in the plane bounded by edges of $G$.
}

The history of inverse degree stems from the conjecture-generating program Graffiti~\cite{Fajtlowicz}. Let $n$ denote $|V|$ and $m$ denote $|E|$. Graffiti posited that the \emph{mean distance} $\tbinom{n}{2}^{-1} \sum_{\{u,v\} \subset V} dist(u, v)$ is always at most the inverse degree $r(G)$. This was disproved by Erd\H{o}s, Pach \& Spencer~\cite{Erdos}, who also proved the tight bound $D = O\bigl(\frac{\log n}{\log \log n} \cdot r\bigr)$ in the process. Subsequently, Mukwembi~\cite{Mukwembi} studied the diameter for various kinds of graphs in terms of inverse degrees. Among other things he conjectured that for any \emph{planar} graph $G$, $D(G)\leq \frac{9}{4}r(G)$.

We disprove Mukwembi's conjecture and establish just how large $D/r$ can be:
\begin{theorem}
\label{thm:MainThm}
For any planar graph $G$, $D(G) < \frac{5}{2}r(G)$. There is an infinite family of graphs with $D(G) = \frac{5}{2}r(G) - O(1)$.
\end{theorem}

\comment{Moreover, we would like to conjecture that the multiplicative constant in the previous theorem should be as follows.
\begin{conjecture}
\label{conj:MainConj}
For any planar graph $G$, $D(G)\leq \frac{5}{2}r(G)$.
\end{conjecture}}

The tight family we construct is very simple, but the bound $D(G) \leq \frac{5}{2}r(G)$ turns out to be quite challenging. A natural approach is to use the arithmetic-harmonic mean inequality to bound $r$ with the simpler quantity $r \ge \frac{n^2}{2m}$; to this end we prove the tight bound $D \le \frac{4n^2}{3m}$ using a simple ``surgery argument."

However, the tight examples of graphs with $D = \frac{4n^2}{3m} - O(1)$ are non-regular (about $2/3$ of vertices have degree 5, and $1/3$ have degree 2) and so they are not tight for the ratio $D/r$ (since our use of the arithmetic-harmonic mean is tight only for regular graphs). Indeed, the bounds $D \le \frac{4n^2}{3m}$ and $r \ge \frac{n^2}{2m}$ do not imply \prettyref{thm:MainThm}, but rather the weaker bound $D \le \frac{8}{3}r$. To actually prove \prettyref{thm:MainThm} (in \prettyref{sec:funtimes}) we carefully engineer a more intricate version of the surgery argument.

\section{Initial Bounds from Surgery}\label{sec:initial}
In this section we focus on proving the less complex bound $D \le \frac{4n^2}{3m}$, and on proving that the ratio $\frac{4}{3}$ is best possible, for connected planar graphs. We use the following sneaky attack on the problem:
\begin{theorem}
\label{thm:AuxThm}
For every connected planar graph, $D\leq \frac{4(n-1)-m}{3}$.
\end{theorem}
We give the proof later in this section, introducing our surgery approach along the way. It gives the desired corollary:
\begin{corollary}\label{corollary:foo}
For every connected planar graph, $D\leq \frac{4n^2}{3m}$.
\end{corollary}
\begin{proof}
We know $(2(n-1)-m)^2 \ge 0$; rearranging yields $4(n-1)-m \le 4\frac{(n-1)^2}{m}$, thus \prettyref{thm:AuxThm} yields
$D(G) \le \frac{4(n-1)-m}{3} \le \frac{4(n-1)^2}{3m}$, which implies the corollary.
\end{proof}

We give some examples before proving \prettyref{thm:AuxThm}. One example disproves Mukwembi's conjecture, and the others demonstrate the tightness of the above theorems. For any even integer $n \ge 4$, let $L_n$ denote the graph with vertices $v^i_j$ for $i \in \{1, 2\}, 1 \le j \le n/2$, such that distinct nodes $v^i_j, v^{i'}_{j'}$ are joined by an edge whenever $|j-j'| \le 1$; the left side of Figure~\ref{figure:examples} illustrates $L_8$. Its diameter is $D(L_n) = n/2-1$, and its inverse degree is $r(L_n) = \frac{n-4}{5}+\frac{4}{3}$. Hence $D = \frac{5}{2}r - O(1)$ for this family of graphs and the second half of \prettyref{thm:MainThm} is proven.

\begin{figure}
\begin{center}
\begin{pspicture}(-0.5,-0.5)(3.5,1.5)
\psset{arrows=*-*,unit=2cm}
\psline(0,0)(0,1)
\psline(0,0)(1,0)
\psline(0,0)(1,1)
\psline(0,1)(1,1)
\psline[linestyle=dashed](0,1)(1,0)
\psline(1,0)(1,1)
\psline(1,0)(2,0)
\psline(1,0)(2,1)
\psline(1,1)(2,1)
\psline[linestyle=dashed](1,1)(2,0)
\psline(2,0)(2,1)
\psline(2,0)(3,0)
\psline(2,0)(3,1)
\psline(2,1)(3,1)
\psline[linestyle=dashed](2,1)(3,0)
\psline(3,0)(3,1)
\end{pspicture}
\hfill
\begin{pspicture}(-1,-0.5)(6.5,2.5)
\psset{arrows=*-*}
\psline(0,0)(2,0)
\psline(-0.5,1)(0,0)
\psline(-0.5,1)(0,2)
\psline(2,0)(-0.5,1)
\psline(-0.5,1)(1.5,1)
\psline(1.5,1)(0,2)
\psline(0,2)(2,2)
\psline[linestyle=dashed](2,2)(0,0)
\psline[linestyle=dashed](0,0)(0,2)
\psline(2,0)(4,0)
\psline(1.5,1)(2,0)
\psline(1.5,1)(2,2)
\psline(4,0)(1.5,1)
\psline(1.5,1)(3.5,1)
\psline(3.5,1)(2,2)
\psline(2,2)(4,2)
\psline[linestyle=dashed](4,2)(2,0)
\psline[linestyle=dashed](2,0)(2,2)
\psline(4,0)(6,0)
\psline(3.5,1)(4,0)
\psline(3.5,1)(4,2)
\psline(6,0)(3.5,1)
\psline(3.5,1)(5.5,1)
\psline(5.5,1)(4,2)
\psline(4,2)(6,2)
\psline[linestyle=dashed](6,2)(4,0)
\psline[linestyle=dashed](4,0)(4,2)
\psline(6,0)(6,2)
\psline(6,2)(5.5,1)
\psline(5.5,1)(6,0)
\end{pspicture}
\end{center}\caption{These planar graphs are depicted as if they were drawn on a cylindrical tube, with the dashed edges hidden on the back. Left: the graph $L_8$. Right: the graph $T_{12}$. }\label{figure:examples}
\end{figure}
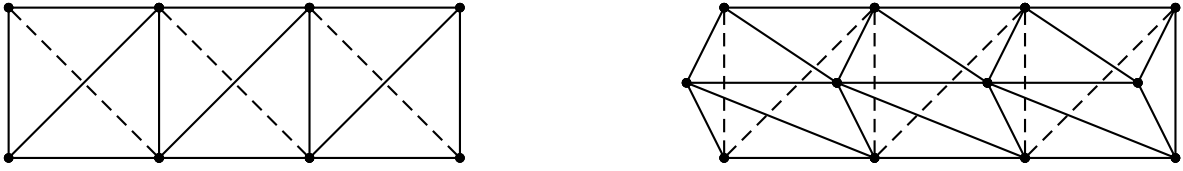

Here is the tight example for \prettyref{corollary:foo}: for any $n$ divisible by 3, take $L_{2n/3}$ and attach a path with $n/3$ additional nodes to $v^1_1$. The resulting graph has diameter $\frac{2n}{3}-1$ and $m = 5\frac{n}{3}-4+\frac{n}{3}$ edges, so $\frac{4n^2}{3m D}$ tends to 1 as $n$ tends to infinity.

Finally, \prettyref{thm:AuxThm} is best possible, up to an additive constant, for all possible values of $m$ and $n$. \emph{Euler's bound} says that in planar graphs having $n \ge 3$, we have $m \le 3n-6$; this maximum is achieved only for triangulations. For $n \ge 6$ divisible by 3, let $T_n$ be obtained from gluing a sequence of $\frac{n}{3}-1$ octahedra at opposite faces; we illustrate $T_{12}$ in the right side of Figure~\ref{figure:examples}. To demonstrate tightness of \prettyref{thm:AuxThm} we start with the extremal values of $m$. For $m = n-1$ we have exact tightness: the path graph $P_n$ has $D(P_n) = n-1 = \frac{4(n-1)-m(P_n)}{3}$. For $m = 3n-6$ when 3 divides $n$, the graph $T_n$ has $D = \frac{n}{3}-1$ and $3n-6$ edges, which is tight for \prettyref{thm:AuxThm} up to an additive constant; other $n$ are similar. More generally, for any $n$ and any $n-1 \le m \le 3n-6$, taking $T_{3 \lceil(m+2-n)/6 \rceil}$ and adding a path of $n- 3 \lceil(m+2-n)/6 \rceil$ more vertices to one end gives an $n$-node, $m$-edge graph with $D = \frac{4(n-1)-m}{3} - O(1)$.

\comment{In our proofs of We use the well-known bound on the number of edges in a planar graph:
\begin{theorem}
\label{thm:Euler}
The number of edges in a planar graph $G$ on $n>2$ vertices is at most $3n-6$.
In case that $G$ is disconnected and each of its  components has at least $3$ vertices, the number of edges of $G$ is at most $3n-12$.
\end{theorem}}

Now we give the proof of \prettyref{thm:AuxThm}, which has some ingredients used later on: a \emph{surgery} operation and decomposition into levels. In the proof, we will let $st$ be a diameter of $G$, e.g.~$dist_G(s, t)=D(G)$.
We let $V_i$, the \emph{$i$th level}, denote all vertices at distance $i$ from $s$, hence $\biguplus_{i=0}^D V_i$ is a partition of $V$. We use the shorthand $V_{[i, j]}$ to mean $\uplus_{i \le x \le j} V_x$ and $V_{\ge i}$ is analogous. Additionally, $G[X]$ denotes an induced subgraph and we will extend the subscript notation on $V$ to mean induced subgraphs of $G$, for example $G_{\ge i}  = G[V_{\ge i}]$.

\begin{proof}[Proof of \prettyref{thm:AuxThm}]
Assume for the sake of contradiction that $G$ is a graph with $D(G) > \frac{4(n-1)-m}{3}$, assume that $n$ is minimal over all such graphs; we may clearly also assume $E$ is \emph{maximal} in the sense that for any $e \not\in E$, either $G \cup \{e\}$ is non-planar or $D(G \cup \{e\}) < D(G)$.

Our first step is to show that $G$ is 2-vertex-connected. Otherwise, pick an articulation vertex $v$, then we can decompose $G$ into graphs $G_1, G_2$ with $V(G_1) \cap V(G_2) = \{v\}, V(G_1) \cup V(G_1) = V(G), E(G_1) \cup E(G_2) = E(G)$, and $n(G_1), n(G_2) <  n(G)$ (a \emph{1-sum}). By our choice of $G$, both $G_i$'s satisfy the conclusion of \prettyref{thm:AuxThm}. Moreover it is easy to see $m(G) = m(G_1)+m(G_2)$ and $D(G) \le D(G_1)+D(G_2)$. Hence
$$D(G)\leq D(G_1)+D(G_2)\leq \tfrac{4(n(G_1)-1)-m(G_1)}{3}+\tfrac{4(n(G_2)-1)-m(G_2)}{3} = \tfrac{4(n(G)-1)-m(G)}{3},$$
contradicting the fact that $G$ was chosen to be a counterexample. Thus $G$ is indeed 2-vertex-connected.

We now consider the diameter $st$ and the level decomposition mentioned previously.
Note that there are no edges between any pair of vertices in $V_i$ and $V_j$ if $|i-j|>1$. It is easy to see that if $|V_i| = 1$ for some $0<i<D$ then $V_i$ is an articulation point, so we have (by 2-vertex-connectivity) that $|V_i| \ge 2$ for all $0 < i < D$.

To begin, suppose $|V_i| \le 2$ for all $i \neq 0$. Since each vertex can only connect to neighbours in $V_{i-1}, V_i, V_{i+1}$ the maximum degree is 5 (and 2 for $s$, 3 for $t$, 4 in $V_1$). Thus (assuming $n \ge 4$ which is easy to justify) we have $D = \lfloor \frac{n}{2} \rfloor$ and $m \le \lfloor \frac{5n-7}{2} \rfloor$, whence it is easy to verify $D \le (4(n-1)-m)/3$ as needed.

Hence, there exists a level of size $\ge 3$. We need one well-known fact and a technical claim.
\begin{fact}\label{fact:2sums}
Let $G_1, G_2$ be planar graphs with $V(G_1) \cap V(G_2) = \{u, v\}$ and $uv \in E(G_1), E(G_2)$. Define their \emph{2-sum} $G$ by $V(G) = V(G_1) \cup V(G_2)$, $E(G) = E(G_1) \cup E(G_2)$. Then $G$ is planar.
\end{fact}
\begin{claim}\label{claim:hooray}
If $|V_i|=2$, $i<D$, then there is an edge joining the two vertices of $V_i$.
\end{claim}
\begin{proof}
Suppose otherwise. Let $V_i = \{u, v\}.$ We will show $uv$ can be added to $G$ without violating planarity, which will complete the proof, since $G$ was chosen edge-maximal (and adding $uv$ does not change $D$).

Since $G$ is 2-vertex-connected, $u$ is not an articulation vertex, so $G[\{v\} \cup V_{>i}]$ is connected, and similarly for $G[\{u\} \cup V_{>i}]$. Thus there is a path $P_R$ from $u$ to $v$ all of whose internal vertices lie in $V_{>i}$. Likewise there is a $u$-$v$ path $P_L$ all of whose internal vertices lie in $V_{<i}$ (e.g.~concatenate shortest $u$-$s$ and $s$-$v$ paths).

Consider a drawing of $G$. The sub-drawing of $G_{\le i}$ must have $u$, $v$ on the same face due to $P_R$, so $G_{\le i} \cup \{uv\}$ is planar. Likewise $G_{\ge i} \cup \{uv\}$ is planar and using \prettyref{fact:2sums}, $G \cup \{uv\}$ is planar as needed.
\end{proof}

Recall there exists a level of size at least 3, let $L$ be chosen minimal with $|V_{L+1}| \ge 3$. Let $R$ be chosen maximal such that $R>L$, and all of the levels $V_{L+1}, V_{L+2}, \dotsc, V_{R-1}$ have size 3. Thus either $R = D+1$, or $R \le D$ and $|V_R| < 3$. We break into several similar cases now.

{\bf Case} $L>0, R < D$. Thus $|V_L|=|V_R|=2$. Consider the graph $G'$ obtained by ``surgery" from $G$ by deleting all edges in $G_{[L, R]}$, deleting the isolated vertices $V_{[L+1, R-1]}$, then adding a clique on $V_L \cup V_R$. This is a planar graph by \prettyref{fact:2sums} and \prettyref{claim:hooray}: it is obtained by two 2-sums from $G_{\le L}$, $K_4$, and $G_{\ge R}$. We illustrate in Figure~\ref{fig:mainPi}. Now $G'$ is smaller than $G$; write $\Delta D = D(G) - D(G'), \Delta m = m(G) - m(G'), \Delta n = n(G) - n(G')$. We have $\Delta n \ge 3 \Delta D$ since all deleted levels had size at least 3. Moreover, since $G_{[L, R]}$ is a planar graph Euler's bound gives that we deleted at most $3(\Delta n + 4) - 6$ edges and added 6 to the new clique, so $\Delta m \le 3 \Delta n$. Thus $\frac{4(\Delta n) - \Delta m}{3} \ge \frac{\Delta n}{3} \ge \Delta D$ and from this it is easy to verify that $G'$ is a smaller counterexample to \prettyref{thm:AuxThm}, contradicting our choice of $G$.

\begin{figure}[h]
\begin{center}
\includegraphics[scale=0.55]{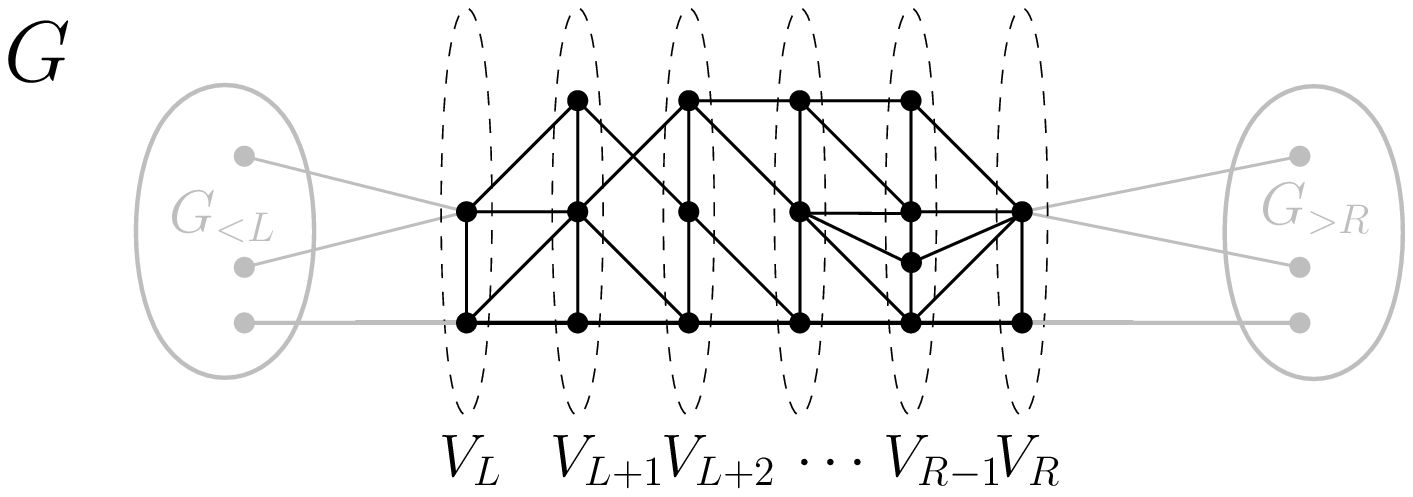}
\hfill
\includegraphics[scale=0.55]{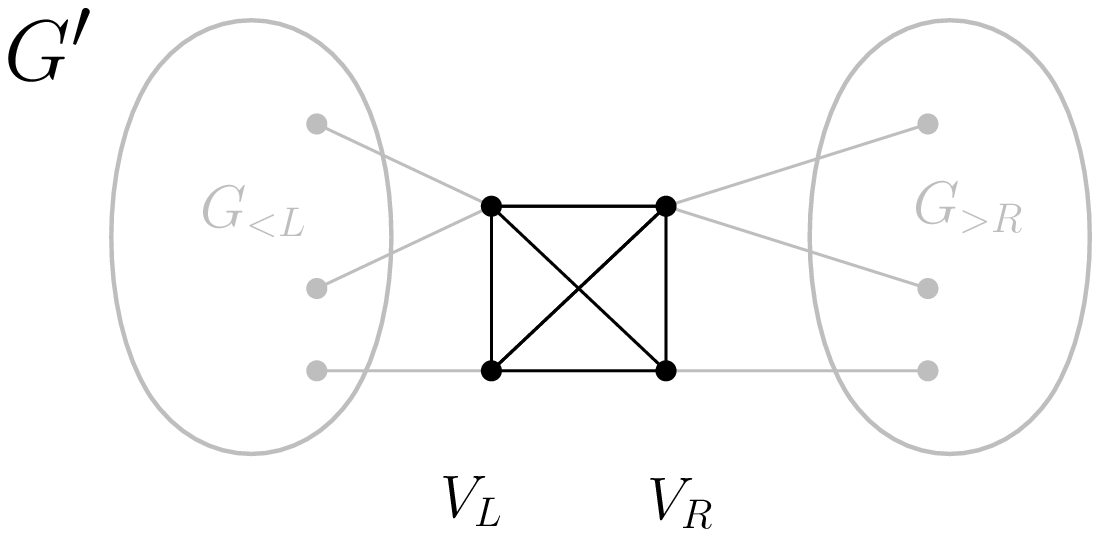}
\end{center}
\caption{Depiction of how surgery changes a graph $G$ (left) into $G'$ (right). Note the $V_i, G_i$ labels are with respect to the original graph. Gray parts are unaltered.}\label{fig:mainPi}
\end{figure}

{\bf Case} $L>0, R \in \{D, D+1\}$. Let $X = V_{> L} \bs \{t\}$. We delete all edges in $G_{\ge L}$, then the isolated vertices $X$, then we join the three vertices $V_L \cup \{t\}$ by a clique. Thus $\Delta m \le 3(\Delta n + 3) - 6 - 3 = 3\Delta n$ and we proceed as before.

{\bf Case} $L=0, R<D$ is the mirror image of the previous case (e.g.~the clique is added to $V_R \cup \{s\}$).

{\bf Case} $L=0, R \in \{D, D+1\}$. We have $n \ge 3D - 1$ since all levels in $V_{[1, D-1]}$ have size at least 3. Using Euler's bound, $4(n-1)-m \ge n+2 > 3D$ and $D < \frac{4(n-1)-m}{3}$ as needed.
\end{proof}

\section{Proof that $r(G) \ge \frac{2}{5}D(G)$ for Planar Graphs}\label{sec:funtimes}
The general idea in the proof of \prettyref{thm:MainThm} is similar to what we did in the previous section, but the devil is in the details, because the terms $1/d(v)$ change in quite complex ways when we perform surgery on the graph. For example, it is no longer possible to easily argue that the selected counterexample $G$ is 2-vertex-connected.
Here is the sketch of how we prove $r(G) \ge \frac{2}{5}D(G)$.
\begin{itemize}
\item Define the \emph{fitness} of a planar connected graph $G$ to be $\fit(G) := \frac{2}{5}D(G) - r(G)$. So we want to show no graph has positive fitness.
\item Let $n$ be minimal such that some $n$-vertex planar connected graph has positive fitness. Subject to this minimal $n$, take such a graph $G$ having maximal fitness. If another graph $G'$ exists such that $|V(G')| \le |V(G)|$ and $\fit(G') \ge \fit(G)$ and at least one of the these two inequalities is strict, this contradicts our choice of $G$. Therefore, the proof strategy uses several parts, and in each part we either find such a $G'$, or impose additional structure on $G$.
\item Let $st$ be any diameter of $G$. We show that except for $s$ and $t$, every vertex has degree at least 3, and that $s$ and $t$ have degree 2 or more.
\item We lay out the graph $G$ in levels, as in the previous proof: level $V_i$ consists of all vertices at distance $i$ from $s$, hence $\uplus_{i=0}^D V_i$ is a partition of $V$.
\item We arrive at a general ``cornerstone" theorem (\prettyref{theorem:cornerstone}) showing that in many cases, a surgery like in \prettyref{sec:initial} finds the desired $G'$.
\item We clean up some additional cases, and thereby prove that $G$ has at most 3 nodes per level, that no size-3 levels are adjacent, that for every size-2 level the contained nodes share an edge, and that the last level $V_D$ has size 1.
\item We use a computation (\prettyref{sec:computation}) to prove that this structured graph has $\fit(G) < 0$, completing the proof.
\end{itemize}

\subsection{Preliminaries}
We reiterate the main tool in the proof.
\begin{claim}\label{claim:useful}
If $G'$ is another graph obtained from $G$, with $n(G') < n(G)$, such that $D(G') \ge D(G) - \Delta D$, $r(G') \le r(G) - \Delta r$, and $\Delta r \ge \frac{2}{5} \Delta D$, then $G'$ is smaller but at least as fit as $G$, contradicting our choice of $G$.
\end{claim}

Since adding an edge decreases $r$ and increases fitness, we also have the following.
\begin{claim}[Maximality]
If $uv \not\in E$ then either $G \cup \{uv\}$ is non-planar or $D(G \cup \{uv\})<D(G)$. In particular, when $u$ and $v$ are in the same levels or adjacent levels, since adding $uv$ would not change the diameter, we have that $G \cup \{uv\}$ is non-planar.
\end{claim}

We will repeatedly make use of the arithmetic-harmonic mean in the following way.
\begin{proposition}\label{proposition:amhm}
For any set $S$ of vertices, $\sum_{v \in S} 1/d(v) \ge |S|^2/(\sum_{v \in S} d(v))$.
\end{proposition}
Thus, the contribution to $r$ by any set is at least as big as what it would give ``on average" by counting all endpoints incident on $S$. Later, we will count $\sum_{v \in S} d(v)$ as twice the number of edges of $G[S]$, plus the number of edges with exactly one endpoint in $S$.

\comment{Said in a false and very rough way, our hope is as follows: the $x$ deleted vertices have average degree ``roughly" 6 or less, hence deleting them decreases $r$ by $x/6$ or more (ignoring what happens at boundary vertices) and decreases the diameter by $x/3$ or less. Thus, roughly, $r(G') - \frac{2}{5}D(G') \le r(G) - x/6 - \frac{2}{5}(D(G)-x/3) \le r(G) - \frac{2}{5}D(G) < 0$ and so $G'$ is a smaller counterexample than $G$, contradicting our choice of $G$ as having a minimal number of vertices. But the boundary effects we are ignoring are significant and take a lot of work to clean up rigorously.}

Suppose that every level of $G$, except possibly the first and last ($V_0$ and $V_D$) have size 3. Then $n \ge 3(D-1) + 2$ and the following proposition shows such graphs are not problematic.
\begin{proposition}\label{proposition:simplecase}
If $n \ge 3(D-1)+2$, then $r(G) \ge \frac{2}{5}D$.
\end{proposition}
\begin{proof}
The case that $|n| < 3$ is easy to verify, so assume $|E| \le 3n-6$. \prettyref{proposition:amhm} applied to $S=V$ implies that $r \ge n^2/(6n-12)$, and by hypothesis $D \le (n+1)/3$. Therefore it is enough to prove $n^2/(6n-12) \ge \frac{2}{5}(n+1)/3$, which is easy to verify by cross-multiplying and solving the resulting quadratic.
\end{proof}

\subsection{Small-Degree Vertices and Articulation Points}
\begin{proposition}
$G$ does not have a degree-1 vertex.
\end{proposition}
\begin{proof}
Let $v$ be a degree-1 vertex with neighbour $z$. We may assume $|V| \ge 3$ so $d(z) \ge 2$. How do $r$ and $D$ change if we get another graph $G'$ by deleting $v$? Clearly $D$ decreases by at most 1; and $r(G') = r(G) - \frac{1}{1} - \frac{1}{d(z)} + \frac{1}{d(z)-1} \le r(G) - 1/2$. In \prettyref{claim:useful} take $\Delta D = 1$ and $\Delta r = 1/2$, we are done.
\end{proof}

A repeated issue is that $r$ is not monotonic, i.e.~sometimes we can decrease $r$ in a graph by adding extra vertices (e.g.~consider the complete bipartite graphs, where $r(K_{2,10} < r(K_{1, 10})$). The following proposition is a first attack against this issue and shows that adding extra blocks (2-vertex-connected components) cannot decrease $r$.
\begin{proposition}\label{proposition:articulation}
If $v$ is an articulation vertex of $G$, then $G \bs v$ has exactly two connected components, one containing $s$ and one containing $t$.
\end{proposition}
\begin{proof}
If the proposition is false, there is an articulation vertex $v$ such that a connected component $H$ of $G \bs \{v\}$ contains neither $s$ nor $t$. Thus $G \bs H$ contains $s$ and $t$, moreover $D(G \bs H) = D(G)$ since any simple $s$-$t$ path goes through $v$ at most once and hence does not use any vertex of $H$.

We want to argue that $r(G \bs H) \le r(G)$, which will complete the proof using \prettyref{claim:useful} with $\Delta D = \Delta r = 0$. It is enough to use very crude degree estimates. Let $|V(H)|=k$. Each vertex of $H$ has degree at most $k$ in $G$ since each $u \in V(H)$ can only have neighbours in $V(H) \cup \{v\} \bs \{u\}$. Moreover, the difference between $r(G \bs H)$ and $r(G)$ is due only to vertices in $\{v\} \cup V(H)$. Clearly $v$ has at least one neighbour not in $H$. Then $$r(G) = r(G \bs H) + \sum_{u \in H} \frac{1}{d_G(u)} + \frac{1}{d_G(v)} - \frac{1}{d_{G \bs H}(v)} \ge r(G \bs H) + \frac{k}{k}+ 0 - 1 =  r(G \bs H),$$
as needed.
\end{proof}

\begin{proposition}\label{proposition:no2}
Except possibly $s$ and $t$, $G$ does not have a degree-2 vertex.
\end{proposition}
\begin{proof}
Let $v \not\in\{s, t\}$ be a degree-2 vertex, with neighbours $a, b$. If $a$ and $b$ are non-adjacent, we can remove $v$ and directly connect them, which decreases $r$ by 1/2 and decreases $D$ by at most 1, which yields a contradiction by \prettyref{claim:useful}.

Therefore assume $a$ and $b$ are adjacent. If both $a$ and $b$ have degree 2 then $G = K_3$ and $\fit(G) < 0$, so we are done. If both $a$ and $b$ have degree at least 3, since $v \not\in \{s, t\}$, $G \bs \{v\}$ is a connected planar graph with diameter at least as large as $G$ and $r(G') \ge r(G) - 1/2 + 1/6 + 1/6 \ge r(G)$, so we are done by using \prettyref{claim:useful} with $\Delta D = \Delta r = 0$.

The final case is that $a$ has degree 2 (w.l.o.g.) and $b$ has degree at least 3. Then $b$ is an articulation vertex, implying by \prettyref{proposition:articulation} that $a \in \{s, t\}$, say w.l.o.g.~$a=s$, and $t \not\in \{v, a, b\}$. But this contradicts edge-maximality in the following way: let $by$ for $y \not\in \{a, v\}$ be an edge on a common face with $bv$ (see Figure \ref{fig:propDeg2}), then adding $vy$ to $G$ does not change the diameter.
\end{proof}

\begin{figure}[h]
\centering
 \subfigure[]{
\includegraphics[scale=0.9]{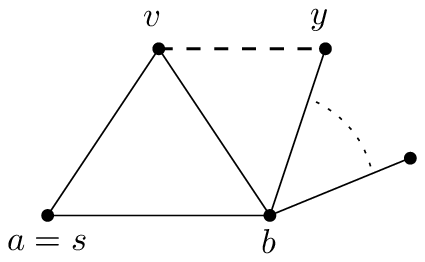}
\label{fig:propDeg2} }  \hspace{10mm}
 \subfigure[]{
\includegraphics[scale=0.9]{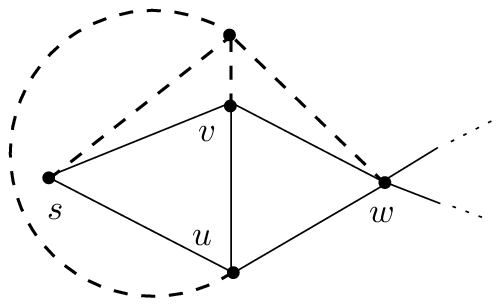}
\label{fig:distance2} }
\caption{Dashed edges are added without violating planarity. (a) The edge $vy$ contradicting the edge-maximality. (b)  The distance 2 neighbourhood of $s$ after $\omega$-$\mu$ surgery and the added edges.}
\end{figure}

\subsection{Basic Surgery: Case Analysis and Bonuses}
The central idea for surgery comes from the first case of \prettyref{thm:AuxThm}'s proof.
\begin{definition}\label{definition:surgery}
Given two levels $V_L$ and $V_R$, to apply \emph{surgery at $V_L$ and $V_R$} means to delete all nodes in $V_{[L+1, R-1]}$ (and their incident edges) and then to connect each $u \in V_L$ to each $v \in V_R$ (we ``add a biclique").
\end{definition}
We say a level of size 2 is \emph{connected} if its vertices share an edge, and that a level of size 1 is always connected. Assuming the levels are connected and of size  at most 2, \prettyref{definition:surgery} is indeed the same surgery as in \prettyref{sec:initial}. As before we get:
\begin{proposition}
Suppose $|V_L|, |V_R| \le 2$ are connected levels with $L < R$. Surgery at $V_L$ and $V_R$ yields a connected planar graph $G'$ with $D(G') = D(G) - (R-L-1)$.
\end{proposition}
We need a collection of \emph{types} (cases) for our analysis. There are 7 types and $V_L$ may satisfy one or none of them (i.e.~the cases are not exhaustive; nonetheless they form the core of our arguments). Analogous cases for $V_R$ are explained afterwards. Here are the 7 types for $V_L$:
\begin{packed_enum}
\item[$\omega$:] $L=0$, i.e.~the level contains one end of the diameter $st$; for all other cases, $L>0$.
\item[$\alpha$:] $|V_L|=1$ and the node in $V_L$ has 1 neighbour in $V_{L-1}$
\item[$\beta$:] $|V_L|=1$ and the node in $V_L$ has 2 neighbours in $V_{L-1}$
\item[$\beta'$:] $|V_L|=1$ and the node in $V_L$ has $\ge2$ neighbours in $V_{L-1}$ and $\ge2$ neighbours in $V_{L+1}$
\item[$\mu$:] $|V_L|=2$, $V_L$ is connected, and each node of $V_L$ has 1 neighbour in $V_{L-1}$, in fact the same one
\item[$\nu$:] $|V_L|=2$, $V_L$ is connected, and each node of $V_L$ has 2 neighbours in $V_{L-1}$
\item[$\nu'$:] $|V_L|=2$, $V_L$ is connected, and each node of $V_L$ has $\ge 2$ neighbours in $V_{L-1}$ and $\ge 2$ neighbours in $V_{L+1}$
    \end{packed_enum}
The analogous cases for the right-hand side are the same with $L=0, L>0$ replaced by $R=D, R<D$, $V_L$ replaced by $V_R$, $V_{L-1}$ replaced by $V_{R+1}$, and $V_{L+1}$ replaced by $V_{R-1}$ (note the sign changes).

Fix $V_L, V_R$ each of size $\le 2$ with $L<R$, such that all levels in between have size at least 3.
Our proof's cornerstone, which we complete at the end of \prettyref{sec:funnycase}, is to show that when $L$ and $R$ are each of one of the 7 types, provided there are at least 4 nodes between $V_L$ and $V_R$, we can get a smaller $G'$ which is at least as fit as $G$, by using surgery and some other ``bonus" operations, contradicting our choice of $G$. After this cornerstone we deal with cases outside the 7 types.

First note that if both $L$ and $R$ are of type $\omega$, \prettyref{proposition:simplecase} already ensures $r(G) \ge \frac{2}{5}D(G)$. If $V_L$ is of type $\lambda$ and $V_R$ is of type $\xi$, we call the surgery type $\lambda$-$\xi$; we call $\omega$-$\omega$ the \emph{unneeded type} of surgery since we don't need to analyze it. It is essential to increase post-surgery fitness when possible. We now establish some values $bonus(\{\lambda,\xi\})$ (which are symmetric in $\lambda$ and $\xi$) such that, after a $\lambda$-$\xi$ surgery, we can increase the fitness by at least $bonus(\{\lambda,\xi\})$.
\begin{packed_enum}
\item We may take $bonus(\{\alpha,\beta\})=bonus(\{\alpha,\beta'\}) = \frac{1}{10}$ because this surgery results in a degree-2 vertex, which may be shortcutted to decrease $D$ by 1 and decrease $r$ by 1/2, giving a $\frac{1}{2}-\frac{2}{5}$ increase in fitness.
\item Similarly we may take $bonus(\{\alpha,\alpha\}) = \frac{2}{10}$.
\item We may take $bonus(\{\omega,\beta\}) = bonus(\{\omega,\beta'\}) = \frac{13}{30}$ as follows. Consider a $\omega$-$\beta$ (or $\beta'$) surgery, so $V_R$ is a singleton $\{v\}$. After surgery $s$ has only one neighbour, $v$, and $v$ has degree at least 3. Then deleting $s$ decreases the diameter by 1 and decreases $r$ by at least $1 - 1/6$. Therefore there is a bonus of at least $1-1/6-2/5 = \frac{13}{30}$.
\item Similarly we can get $bonus(\{\omega,\alpha\}) = 13/30+1/10 = 8/15$ because (w.l.o.g.~in a $\omega$-$\alpha$ surgery) the $\alpha$ vertex's right neighbour has degree at least 3 in the original and post-operation graphs, using \prettyref{proposition:no2}.
\item Finally we can get $bonus(\{\omega,\mu\}) = 1/12$ as follows. Consider a (w.l.o.g.) $\mu$-$\omega$ surgery, where $V_L = \{u, v\}$ and the common neighbour of $u, v$ in $V_{L-1}$ is $w$. Post-surgery, the distance-2 neighbourhood of $s$ is as shown in Figure \ref{fig:distance2}. Add a new vertex and connect it to $u, v, w, s$; it is not hard to argue this preserves planarity. Not counting the increased degree at $w$, we decreased $r$ by $\frac{1}{2}+\frac{2}{3}-\frac{1}{3}-\frac{3}{4}=\frac{1}{12}$ and preserved $D$. (Although this adds a vertex, the surgery theorems later on always delete at least 2 vertices, so overall the total number of vertices always decreases.)
\end{packed_enum}

\subsection{First Analysis of Surgery}
Now we give a lower bound on fitness increase due to surgery. It is convenient to assume when $V_L$ is in cases $\beta', \nu'$ that each node in $V_L$ has \emph{exactly} two neighbours in $V_{L-1}$ --- call the rest \emph{ghost neighbours}. Why is this ok? Keep in mind we want to lower bound the fitness increase from surgery. Due to the ``$\ge2$ neighbours in $V_{L+1}$" condition in these cases, surgery does not increase the degree of nodes in $V_L$. Further, by the convexity of $d(v) \mapsto \frac{1}{d(v)}$, the actual $r$ increase including ghost neighbours will be no more than the ``virtual $r$ increase" ignoring ghost neighbours made by our analysis.

Here are the details. Let $n_L$ denote $|V_L|$ and similarly for $n_R$. Let $o_L$ denote, for each node in $V_L$, the number of ``outside" neighbours such nodes have in $V_{L-1}$; define $o_R$ similarly with $V_{R+1}$ in place of $V_{L-1}$. Thus $n_L$ and $o_L$ depend only on the type of $L$, and abusing notation, we write $n_\omega = n_\alpha = n_\beta = n_{\beta'} = 1, n_\mu = 2, n_\nu = n_{\nu'}= 2$ and  $o_\omega = 0, o_\alpha = 1, o_\beta = o_{\beta'} = 2, o_\mu = 1, o_\nu = o_{\nu'} = 2$. Let $\overline{o}$ denote the number of neighbours each vertex of $V_L$ has in $V_L \cup V_{L-1}$, so $\overline{o} = o + (n-1)$.
Let $w = R-L-1$ denote the number of levels in between, and recall that each of these $w$ levels has at least 3 nodes. Let $x$ denote the number of nodes in the deleted levels, hence we have $x \ge 3w$.

Before surgery, the sum of the degrees of the nodes in $V_{[L, R]}$ is at most $n_Lo_L + 2(3(n_L+x+n_R)-6) + n_Ro_R$ --- the terms count edges from $V_{L-1}$ to $V_L$, in $G_{[L, R]}$, and from $V_R$ to $V_{R+1}$ respectively. We thereby use \prettyref{proposition:amhm} to lower-bound the initial sum of the inverse degrees in $V_{[L, R]}$. Post-surgery, we know the degrees of the nodes in $V_L$ are $\overline{o}_L + n_R$ and similarly for $V_R$. Therefore, if $G'$ indicates the result of applying surgery and bonus operations, we have $\fit(G') - \fit(G) \ge \eqref{eq:mega}$ defined by
\begin{equation}
\frac{(n_L + x + n_R)^2}{n_Lo_L + 2(3(n_L+x+n_R)-6) + n_Ro_R} - \frac{n_L}{\overline{o}_L + n_R}
- \frac{n_R}{\overline{o}_R + n_L} + bonus(L, R) - \frac{2}{5}w.\tag{$\ast$}\label{eq:mega}
\end{equation}
It is easy to verify that $\eqref{eq:mega} > \frac{x}{6} - 4 - \frac{2}{5}w \ge \frac{x}{6} - \frac{2x}{15} - 4$ so it is clearly positive for $x \ge 120$.
In fact the following precise statement is true and gives what we want in almost all needed cases; we also need some $w=0$ cases for later even though they don't make sense in the context provided above.
\begin{claim}
Let $x, w$ be integers with $x \ge 3w, x \ge 2, w \ge 0$. Except for $(w, x) \in \{(1, 3), (2, 6)\}$, the value \eqref{eq:mega} is positive for all types of $L, R$ (except the unneeded $L=R=\omega$).
\end{claim}
\begin{proof}We use a publicly posted {\tt{Sage}} worksheet~\cite{calculations} to verify the needed cases. (Note we've chosen things so that a $\lambda$-$\xi$ surgery has the same analysis as a $\xi$-$\lambda$ surgery, and such that the pairs $\{\beta, \beta'\}$ and $\{\nu, \nu'\}$ are analyzed in the same way. So our computation involves 14 surgery cases.)\end{proof}
\comment{(Also \eqref{eq:mega} $\ge 0$ for $w=0$ and $x \in \{1, 2\}$ but this is not needed.)
Note also that the $\omega$-$\omega$ case is not needed, due to \prettyref{proposition:simplecase}.}
More generally, the exact same proof gives the following generalization, which is needed later.
\begin{theorem}\label{theorem:gensurg}
Let $V'_R \subseteq V_R$, $L<R$, so that every $s$-$t$ path intersects $V'_R$. Let $X$ be the nodes not connected to $s$ or $t$ in $G \bs V_L \bs V'_R$ and let $x=|X|$. Let $V_L$ be any of the 7 types. Let $V'_R$ be of one of the 7 types, modified so that ``in $V_{R-1}$" is replaced by ``in $X$" and ``in $V_{R+1}$" is replaced by ``in $V_{R+1} \bs X$." Assume that at least one of $L, R$ is not of type $\omega$. Let $w = R - L - 1$. If we delete $X$ and connect $V_L$ to $V'_R$ by a biclique, then perform bonus operations, we get a smaller graph at least as fit as $G$, provided $w\ge0, x\ge2$, $x \ge 3w$ and $(w, x) \not\in \{(1, 3), (2, 6)\}$.
\end{theorem}

\subsection{Completing the Cornerstone: The Case $w = 2, x = 6$}\label{sec:funnycase}
If $w = 2, x = 6$ then $R = L+3$ and $|V_{L+1}| = |V_{L+2}| = 3$, since all levels between $V_L$ and $V_R$ have size at least 3. We need:
\begin{claim}\label{claim:a23}
Let $V_i$ be a level of size 2, whose vertices are connected by an edge, and let $j = i+1$ or $j = i-1$, with $|V_j| = 3$. Then the two vertices of $V_i$ do not have three common neighbours in $V_j$.
\end{claim}
\begin{figure}[h]
\centering
 \subfigure[]{
\includegraphics[scale=0.9]{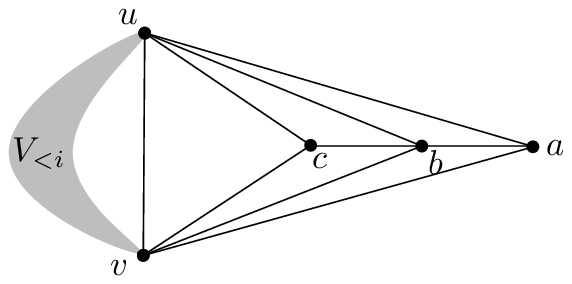}
\label{fig:Claim23} }  \hspace{10mm}
 \subfigure[]{
\includegraphics[scale=0.9]{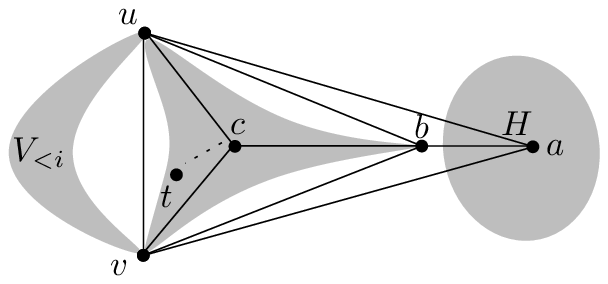}
\label{fig:Claim23-2} }
\caption{(a) If we delete $uvb$ the remainder will have at least 3 connected components. (b) One of these connected components, $H$, does not contain $s$ or $t$; we will delete it.}
\end{figure}
\begin{proof}
The goal of the proof is similar to the result in \prettyref{proposition:articulation}: assume the opposite for the sake of contradiction, then show there is some part of the graph that can be deleted while decreasing $r$ and leaving $D$ unchanged. To do this, we need to establish some structure.

Let $V_i = \{u, v\}$. To simplify the notation we handle the case $j=i+1$ but the proof of the other case is identical. Since $G_{\le i}$ is planar we can draw
it with the edge $uv$ on the outer face. Likewise, draw $G_{\ge i}$ with $uv$ on the outer face. Each vertex of $V_j$ forms a triangle with $uv$
so for some labelling $V_j = \{a, b, c\}$, the drawing of $G_{\ge i}$ has triangle $uva$ containing vertex $b$ and triangle $uvb$ containing vertex $c$, as pictured in
Figure \ref{fig:Claim23}.

We claim by maximality $ab$ is an edge of $G$: indeed, since $u$ has no neighbours other than $v, a, b, c$ in
 the drawing of $G_{\ge i}$, if $ab$ is not present we can add it in a planar way by going next to the path $aub$. Similarly $bc \in E(G)$.

Now note that $G \bs \{u, b, v\}$ has at least 3 components: one containing $a$, one containing $c$, and one containing
$V_{< i}$. One of the first two does \emph{not} contain $t$. Assume the first (the second case is analogous): denote the component
 containing $a$ in $G \bs \{u, b, v\}$ by $H$, so $H \not\ni t$ (see Figure \ref{fig:Claim23-2}). It's not hard to see any shortest $s$-$t$ path avoids $H$, hence $D(G \bs H) = D(G)$.
Moreover we claim $r(G \bs H) < r(G)$, contradicting our choice of $G$. To see this, let $k$ denote $|V(H)|$, note that each vertex in $H$ has degree at most
$k+2$, and that we drop the degrees of $u, b, v$ by at most $k$, thus
\begin{align*}r(G) - r(G \bs H) &\ge  k\frac{1}{k+2} + \sum_{i \in \{u, b, v\}} \frac{1}{\deg_G(i)} - \frac{1}{\deg_{G\bs H}(i)}
\\&\ge
\frac{k}{k+2} + \sum_{i \in \{u, b, v\}} \frac{1}{\deg_{G\bs H}(i)+k} - \frac{1}{\deg_{G\bs H}(i)} \\ &\ge \frac{k}{k+2} + 3(1/(k+3)-1/3) = \frac{k}{(k+2)(k+3)} > 0\end{align*}
where in the second-to-last inequality we used the fact that $\deg_{G\bs H}(i) \ge 3$ and $\frac{1}{\cdot}$ is convex.
\end{proof}
This allows us to bound the number of edges between a level $\{u, v\}$ with $uv \in E$ and an adjacent level of size 3: there are at most 5. It's also obvious that between a singleton level and an adjacent level of size 3, there are at most 3 edges. Accordingly, let $z_L$ be 3 (resp.~5) when $n_L$ is 1 (resp.~2) and similarly define $z_R$.
In the situation that there are exactly two levels, each of size-3, between $V_L$ and $V_R$,
we can replace the quantity $\eqref{eq:mega}$ from the previous section by grouping the vertices in a different way; specifically we have $\fit(G') - \fit(G) \ge \eqref{eq:mega2}$ with \eqref{eq:mega2} defined by
\begin{equation}
\frac{n_L^2}{n_L\overline{o}_L + z_L} +
\frac{x^2}{z_L + 2(3x-6) + z_R} + \frac{n_R^2}{n_R\overline{o}_R + z_R} - \frac{n_L}{\overline{o}_L + n_R}
- \frac{n_R}{\overline{o}_R + n_L} + bonus(L, R) - \frac{2}{5}w.\label{eq:mega2}\tag{\maltese}
\end{equation}
Specifically, the first three terms lower-bound the contribution to $r(G)$ by vertices in $V_L$, in $V_{L+1} \cup V_{L+2}$, and $V_{L+3}=V_R$ respectively.
\begin{claim}
The quantity \eqref{eq:mega2} is positive when $w=2, x=6$ for all types of $L, R$ (except the unneeded type $L=R=\omega$).
\end{claim}
\begin{proof}This calculation is also done via computer at \cite{calculations}.\end{proof}
\begin{corollary}\label{corollary:othercase}
Let $V_L$ and $V_R$ be levels of one of the 7 types (except the unneeded type $L=R=\omega$), with $R = L + 3$ and $|V_{L+1}| = |V_{L+2}| = 3$.
Applying surgery at $V_L$ and $V_R$ gives a smaller which is smaller and more fit than $G$.
\end{corollary}
Together with \prettyref{theorem:gensurg} this gives the heart of our proof:
\begin{theorem}[Cornerstone]\label{theorem:cornerstone}
Let $V_L, V_R$ be levels of size $\le 2$, with all levels between them of size $\ge 3$. If $V_L$ and $V_R$ are each one of the $7$ types, and there are at least $4$ nodes between them, this contradicts our choice of $G$.
\end{theorem}

\subsection{Sufficiency of the 7 Cases}
The structure we want to establish in $G$ is that every level has size at most 3, and that two size-3 levels are never adjacent. We now show how to get from the cornerstone (\prettyref{theorem:cornerstone}) to this structure. We start with a general observation (which motivated our definition of the 7 cases).

\begin{claim}\label{claim:maxy}
Suppose $V_i = \{u, v\}$ and $uv \in E$. Suppose $j = i \pm 1$, that $u$ has 1 or fewer neighbours in $V_j$, and that $v$ has at least one neighbour in $V_j$ which is not a neighbour of $u$. Then this violates maximality.
\end{claim}
\begin{proof}
Take $j=i+1$, the other case is analogous. Embed $G_{\ge i}$ with $uv$ on the outer face. First if $u$ has no neighbours in $V_{i+1}$ then note $u$ and a neighbour of $v$ are on the outer face, hence we can add an edge between them without violating planarity in $G_{\ge i}$ (and hence without violating planarity in $G$, by \prettyref{fact:2sums}). Second, suppose $u$ has exactly one neighbour $x$ in $V_{i+1}$; at least one endpoint emanating from $v$ adjacently to $vu$ is of the form $vy$ with $y \neq u, v, x$; then the path $uvy$ lies on a face and the edge $uy$ can be added without violating planarity.
\end{proof}

In the remainder of the section, we ensure all size-2 levels are connected, show that $V_L$ always is in one of the 7 cases, deal with $V_R$'s that fall outside the 7 cases, and then show the last level $V_D$ has size 1.

\begin{claim}\label{claim:level2con}
Any level of size $2$ is connected, except possibly for the last level $V_D$.
\end{claim}
\begin{proof}
Let $V_R$ be minimal, $R < D$, such that $V_R = \{u, v\}$ is of size 2 and $uv$ is not an edge. If both $u$ and $v$ are connected to $t$ in $G_{\ge R}$ then using the proof method of \prettyref{claim:hooray}, $uv$ can be added without violating planarity, which contradicts maximality. Therefore assume only $u$ has a path to $t$ in $G_{\ge R}$. It now follows that $v$ is an isolated vertex in $G_{\ge R}$,
or else \prettyref{proposition:articulation} is violated because of the articulation point $v$.

Since $v$ has degree at least 3 (by Proposition \ref{proposition:no2}) and these neighbours are only in $V_{R-1}$, it follows that $|V_{R-1}| \ge 3$. Let $L$ be maximal with $L < R$ such that $|V_L| \le 2$. By our choice of $R$, we see $V_L$ is connected if it has size 2. Moreover, each vertex in $V_L$ has at least two neighbours in $V_{L+1}$, using $|V_{L+1}| \ge 3$ and \prettyref{claim:maxy}. So $V_L$ is of one of the 7 cases.

Now look at $u$. If $u$ has 2 or more neighbours in $V_{R-1}$, we can use surgery at $V_L$ and $u$ which is of type $\beta'$ (\prettyref{theorem:gensurg}: cutting out $R-L-1 \ge 1$ levels of size 3, plus $v$). Otherwise, we can use surgery at $V_L$ and the unique neighbour of $u$ in $V_{R-1}$, which is an articulation vertex of type $\alpha$ (\prettyref{theorem:gensurg}: cutting out $R-L-2 \ge 0$ levels of size 3, plus $v$ and at least two nodes from $V_{R-1}$).\end{proof}

The following corollary follows from the previous proof and induction:
\begin{corollary}\label{corollary:leftok}
Every level $V_L$ such that $|V_L| \le 2, |V_{L+1}| \ge 3$ falls in one of the 7 cases.
\end{corollary}
\comment{\begin{proof}
Given the result of \prettyref{corollary:leftok}, this is just repeating the combination of \prettyref{claim:maxy} and $|V_{L+1}| \ge 3$ from the proof of \prettyref{corollary:leftok}.
\end{proof}}
\begin{proposition}
Let $V_R$, $R < D$, be such that $|V_R| \le 2$, and either $|V_{R-1}| \ge 4$, or both $|V_{R-2}|, |V_{R-1}| \ge 3$. Then we can perform surgery to increase the fitness of $G$.
\end{proposition}
\begin{proof}
Let $L<R$ be maximal with $|V_L| \le 2$. Using \prettyref{corollary:leftok} (along with \prettyref{corollary:othercase} or \prettyref{theorem:gensurg}) we may assume $V_R$ falls outside of the 7 types; using \prettyref{claim:level2con} and \prettyref{claim:maxy} this means that either $|V_R|=1$ and it has one neighbour in $V_{R-1}$ but $\ge 3$ neighbours in $V_{R+1}$, or $|V_R|=2$ and these vertices each have one neighbour (the same one) in $V_{R-1}$ and one vertex of $V_R$ has $\ge 3$ neighbours in $V_{R+1}$.

In either case, only one vertex in $V_{R-1}$, call it $v$, is adjacent to $V_R$. Since $v$ is an articulation vertex we can do surgery on $V_L$ and $v$ --- we apply \prettyref{theorem:gensurg} to levels $L$ and $R' = R-1$, on sets $V_L$ and $V'_{R'} = \{v\}$ (here $V'_{R'}$ is of type $\alpha$ if $|V_R|=1$ or $\beta$ if $|V_R|=2$). The set $X$ is $V_{[L+1, R-1]} \bs \{v\}$, and $w = R'-L-1$ so $x=|X| \ge 3w+2, w \ge 0$. This indeed satisfies the conditions of \prettyref{theorem:gensurg} so we are done.
\end{proof}

\begin{proposition}
The size of the last level $V_D$ is 1.
\end{proposition}
\begin{proof}
Suppose $|V_D| > 1$ for the sake of contradiction. Let $V_L$ be the rightmost level of size at most 2, which we know is one of the 7 types by \prettyref{corollary:leftok}. Let $v \in V_D \bs \{t\}$.
 If $L = D-1$ then it is not hard to see some face contains $v$ and a vertex from
 $V_{D-2}$; adding an edge between this pair does not decrease the diameter,
 so contradicts edge-maximality. Otherwise ($L<D-1$) apply surgery to $V_L$ and $t$: we cut out 1 or more levels of size at least 3, plus the vertices of $V_D \bs \{t\}$. Thus $x \ge 3w+1, w \ge 1$ and \prettyref{theorem:gensurg} is satisfied.
\end{proof}
Combining the results just proven, we have the desired structure theorem: $G$ is a graph where the first and last level have size 1, all levels have size at most 3, every level of size 2 is connected, and no two levels of size 3 are adjacent.

\subsection{Computation}\label{sec:computation}
We finish by showing that our hypothetical $G$ has $r \ge \frac{2}{5}D$.
\begin{theorem}\label{theorem:awesome}
Let $G$ be a graph where the first and last level have size 1, all levels have size at most 3, every level of size 2 is connected, and no two levels of size 3 are adjacent. Then $r(G) \ge \frac{2}{5}D + \frac{37}{60}$.
\end{theorem}
\begin{proof}
The most important fact about the structure is that, given the sizes of levels $i-1, i, i+1$, we can determine (or upper bound, depending on how you look at it) the degrees of the nodes in level $i$, which we use to get a lower bound on the sum of the inverse degrees for that level.

Given any two adjacent levels, we may upper bound the edges they share by a biclique. Furthermore, if a level of size 2 and a level of size 3 are adjacent, by \prettyref{claim:a23} we can upper bound their shared edges as being one edge short of a biclique. Hence let $\mathcal{S}(i, j) = i \cdot j$ unless $\{i, j\}=\{2, 3\}$ in which case $\mathcal{S}(i, j)=5$. Thus:
\begin{itemize}
\item $\sum_{v \in V_0} 1/d(v) \ge 1/|V_1|$
\item $\sum_{v \in V_D} 1/d(v) \ge 1/|V_{D-1}|$
\item For $0<i<D$ there are at most $\mathcal{E} := \mathcal{S}(|V_{i-1}|, |V_i|)+2\tbinom{|V_i|}{2}+\mathcal{S}(|V_i|, |V_{i+1}|)$ endpoints incident on $V_i$; considering the degrees are integral and using convexity we see $$\sum_{v \in V_i} 1/d(v) \ge \frac{\mathcal{E} \bmod |V_i|}{\lceil \mathcal{E}/|V_i| \rceil} + \frac{|V_i| - (\mathcal{E} \bmod |V_i|)}{\lfloor \mathcal{E}/|V_i| \rfloor} =: \mathcal{C}.$$
\end{itemize}
Since $\mathcal{C}$ is determined only by $|V_{i-1}|, |V_i|, |V_{i+1}|$, we write it as $\mathcal{C}(|V_{i-1}|, |V_i|, |V_{i+1}|)$. We therefore deduce for any sequence $(n_0, n_1, \dotsc, n_D)$ of level sizes of a graph $G$ that $$r(G) \ge \mathcal{R}(n_0, n_1, \dotsc, n_D) := 1/n_1 + 1/n_{D-1} + \sum_{i=1}^{D-1} \mathcal{C}(|V_{i-1}|, |V_i|, |V_{i+1}|).$$

Finally, we want to determine which valid sequence minimizes $\mathcal{R}(n_0, n_1, \dotsc, n_D) - \frac{2}{5} D$. Because $\mathcal{C}$ is a sum of local contributions, and because each level contributes 1 to the diameter, we can think of this last step as shortest path problem, as follows. Define a new digraph with vertex set $$\{s, (1, 1), (1, 2), (1, 3), (2, 1), (2, 2), (2, 3), (3, 1), (3, 2), t\},$$
where the $(i, j)$-vertices represent a pair of adjacent levels, $s$ represents the start, and $t$ the end. The intuition: we insert an arc from $(i, j)$ to $(k, \ell)$ whenever $j = k$, representing three consecutive levels. The cost of such an edge should account for the $r$-contribution of the level corresponding to $j$, minus the contribution from lengthening the diameter.

Formally, we add an arc $(i, j) \rightarrow (j, k)$ for all $i, j, k$ (with no consecutive 3s) having cost $\mathcal{C}(i, j, k) - \frac{2}{5}$; we add an arc $s \rightarrow (1, i)$ for all $i$ having cost $1/i$; and we add an arc $(i, 1) \to t$ for all $i$ having cost $1/i - \frac{2}{5}.$ Then it's easy to see that for any sequence of $n_i$'s, $\mathcal{R} - \frac{2}{5} D$ is given by the cost of the $(D+1)$-edge path $s \to (n_0, n_1) \to (n_1, n_2) \to \dotsb (n_{D-1}, n_D) \to t$ in the new digraph. Executing a shortest-path algorithm such as Bellman-Ford (see the worksheet at \cite{calculations}) establishes that the shortest path from $s$ to $t$ has cost $\frac{37}{60}$, hence $r \ge \mathcal{R} \ge \frac{2}{5} D + \frac{37}{60}$ for these graphs (and that there are no negative dicycles).
\end{proof}
In fact $r \ge \frac{2}{5} D + \frac{37}{60}$ holds for all graphs, is best possible, and the unique graph with $r = \frac{2}{5}D + \frac{37}{60}$ is $K_5^-$. To establish this precise result, small adjustments to our proofs are necessary, as well as exhaustive searching on all planar graphs with up to 9 vertices.

\section{Conclusion}
The main techniques underlying our diameter bounds for planar graphs were the surgery operation (which preserves planarity), and the fact that every planar graph has at most a linear number of edges. One might try the same approach on the family of graphs excluding any fixed $k$-clique minor, since such graphs have $O(nk \sqrt{\log k})$ edges (e.g., see~\cite{Thomason}).
A perpendicular avenue for future research would be to find a tight relation in connected planar graphs between the mean distance and the diameter.



\end{document}